\documentclass[twoside]{article}
\usepackage[english]{babel}
\usepackage[utf8]{inputenc}
\usepackage{color}
\usepackage[colorlinks=true,citecolor=green]{hyperref}
\usepackage{amsmath,amsfonts,amssymb,amsthm,mathtools}
\usepackage[margin=1in]{geometry}
\usepackage[margin=0.2in]{caption}
\usepackage{graphicx,subfig}
\usepackage{tikz}
\usepackage{authblk}
\usepackage{fancyhdr}
\newcommand{\orcid}[1]{\href{https://orcid.org/#1}{\includegraphics[width=10pt]{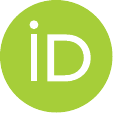}}}

\def\be{\begin{equation}}
	\def\ee{\end{equation}}
\def\bse{\begin{subequations}}
	\def\ese{\end{subequations}}

\newtheorem{theorem}{Theorem}[section]
\newtheorem{corollary}[theorem]{Corollary}
\newtheorem{lemma}[theorem]{Lemma}

\newtheorem{remark}[theorem]{Remark}

\numberwithin{equation}{section}

\def\N{\mathbb{N}}
\def\R{\mathbb{R}}
\def\C{\mathbb{C}}

\def\Z{\mathbb{Z}}

\def\Re{\mathop{\rm Re}\nolimits}
\def\Im{\mathop{\rm Im}\nolimits}

\def\@#1{{\mathbf{#1}}}
\def\_#1{{\mathsf{#1}}}
\let\~=\tilde
\let\==\overline
\let\^=\hat
\let\<=\langle
\let\>=\rangle

\def\max{\mathop{\rm max}\nolimits}

\def\p{\partial}

\def\T{\mathbb{T}}

\def\1{{\bf 1}}

\makeatother

\title{\textbf{Spectral enclosure estimates for non-self-adjoint Dirac operators}}

\author{\Large Jeffrey Oregero\orcid{0000-0003-1576-8867}\footnote{E-mail address:~\href{mailto:oregero@ku.edu}{oregero@ku.edu}
}}
\affil{\normalsize Department of Mathematics, University of Kansas\\
Lawrence, KS 66045, USA}

\date{}

\begin{document}
\maketitle

\noindent\small{2020~\textit{Mathematics Subject Classification.} Primary 34L40, 47A10, Secondary 35Q51}
\vspace{4mm}
	
\begin{abstract}
We study the spectrum of a periodic non-self-adjoint Dirac operator, and its dependence on a semiclassical parameter is also considered. Several bounds on the spectrum are obtained which provide sharp spectral enclosure estimates. Importantly, a sufficient condition is obtained which ensures semiclassical spectral confinement to the real axis and a bounded subset of the imaginary axis of the spectral plane.
\end{abstract}
	

\section{Introduction}
In this work we study the non-self-adjoint Dirac operator
\be
\label{e:dtype}
\mathfrak{D}:= ih\sigma_3\p_x + Q(x), \quad x\in\R,
\ee  
where
\be
\sigma_3 = \begin{pmatrix} 1&0 \\ 0&-1 \end{pmatrix}, \qquad Q(x)=\begin{pmatrix} 0&-iq(x) \\ ip(x)&0 \end{pmatrix}.
\ee
Here $Q:\R\to \C^{2\times 2}$ is a potential, $h>0$ the semiclassical parameter, and $i$ the imaginary unit. We assume $p$, $q$ are $1$-periodic complex-valued functions. Moreover, unless stated otherwise $p$, $q\in L_{\rm loc}^{1}(\R)$, the space of locally Lebesgue integrable functions. Letting $h\to 0^+$ is referred to as the \textit{semiclassical limit}. 

Our aim is to derive sharp spectral enclosure estimates of such operators in the space $L^2(\R)^2$ of $2$-component complex vector functions $\Psi(x;h)=(\psi_1(x;h),\psi_2(x;h))^{T}$ (``$T$'' denoting transpose) with the norm
\be
\label{e:L2vecnorm}
\|\Psi\|_{L^2(\R)^2} = \left(\int_{-\infty}^{\infty} \left(|\psi_1(x;h)|^2 + |\psi_2(x;h)|^2\right)\,dx\right)^{1/2}. 
\ee
Under the above assumptions we can uniquely associate the densely and maximally defined closed linear operator $H$ in $L^2(\R)^2$ with the differential expression \eqref{e:dtype}, i.e.,
\be
\label{e:operator}
H\Psi = \mathfrak{D}\Psi, \quad {\rm dom}(H) = \left\{\Psi\in L^2(\R)^2: \Psi\in AC_{\rm loc}(\R)^2,\, H\Psi\in L^2(\R)^2\right\}.
\ee
where $AC_{\rm loc}(\R)^2$ is the space of locally absolutely continuous 2-component complex vector functions (cf.~\cite[Chapter~4]{BES2013}).

Dirac operators are fundamental objects in mathematical physics finding applications in relativistic quantum mechanics, nonlinear optics, and material science. For a thorough introduction to the Dirac equation see \cite{Thaller} and references therein. 

A major motivation for studying \eqref{e:dtype} is due to its connection with hierarchies of completely integrable nonlinear evolution equations in $1+1$ dimensions. Introducing a deformation parameter $(p(x),q(x))\mapsto (p(x,t),q(x,t))$ the AKNS-hierarchy of nonlinear evolution equations is then constructed via appropriate Lax pairs $(P_{n+1}(t),\mathfrak{D}(t))$, where $P_{n+1}(t)$ is a $2\times 2$ matrix-valued differential expression of order $n+1$, $n\in\N_o:=\N\cup\{0\}$. The collection of Lax equations
\be
\label{e:Laxeq}
{\rm AKNS}_n(p,q;h) := \frac{d}{dt}\mathfrak{D}(t) - [P_{n+1}(t),\mathfrak{D}(t)] = 0,
\ee
where $[\cdot,\cdot]$ denotes the commutator, then defines the AKNS-hierarchy for $(p(x,t),q(x,t))$ (cf.~\cite{GW_acta1998} for an explicit construction of $P_{n+1}$). Importantly, time deformations which satisfy \eqref{e:Laxeq} form isospectral level sets in an appropriate function space. Many equations of fundamental physical importance such as the Korteweg-de Vries (KdV), modified Korteweg-de Vries (mKdV), nonlinear Schr\"odinger (NLS), and sine-Gordon (sG) equations are members of the AKNS-hierarchy.   

The AKNS-hierarchy is named after Ablowitz, Kaup, Newell, and Segur who in a seminal work in 1974 introduced a new system of nonlinear evolution equations, the AKNS-system (cf. \cite{AKNS} for details)
\be
\label{e:aknssys}
{\rm AKNS}_1(p,q;h) = \begin{pmatrix} ih \p_tp - \frac{h^2}{2}\p_x^2p + p^2q \\ ih \p_tq + \frac{h^2}{2}\p_x^2q - q^2p \end{pmatrix} = 0.
\ee
Importantly, the authors showed that the AKNS-system \eqref{e:aknssys} admits a Lax pair and is thus formally integrable using the inverse scattering transform (IST) \cite{APT2004,AKNS,BBEIM,NMPZ1984}.

\begin{remark}
Suppose that $p(x,t)=\pm \overline{q(x,t)}$, where overbar denotes complex conjugation. In this case the AKNS-hierarchy reduces to the nonlinear Schr\"odinger (NS)-hierarchy. In particular, the AKNS-system reduces to the cubic nonlinear Schr\"odinger equation, namely, 
\be
\label{e:nls}
{\rm NLS}_{\pm}(q;h) := ih\p_tq + \frac{h^2}{2}\p_x^2q \mp |q|^2q = 0,
\ee
$(``\pm'' defocusing/focusing, respectively)$ where $q:\R\times \R\to \C$ is the slowly-varying complex envelope of a weakly dispersive nonlinear wave packet, and the physical meaning of the variables depends on the context. (E.g., in nonlinear fiber optics, $t$ represents propagation distance while $x$ is a retarded time.) See, for example, \cite{ADK,AA,FT1987,HK,mclaughlinoverman,solli}. In general, the parameter $h$ quantifies the relative strength of dispersion compared to nonlinearity. (In the quantum-mechanical setting $h$ is proportional to Planck's constant $\hbar$.) Letting $h\to 0^+$ in \eqref{e:nls} is referred to as the ``semiclassical'', or ``zero-dispersion'' limit. These situations can produce a wide variety of physical effects such as supercontinuum generation, dispersive shocks, wave turbulence, and soliton gases, to name a few (e.g., see \cite{BELOT,DT,EH,eltovbisSG,RWOP,TW,ZabuskyKruskal,Z} and references therein). Moreover, see \cite{bertolatovbis,BLOT,BiondiniOregero,pre2017deng,ForestLee,fujiiewittsten2018,JM,JLM,KMM2003,miller2001,TVZ2004} and references therein. 
\end{remark}

An important class of equations associated with the AKNS-hierarchy are the stationary AKNS-equations, characterized by
\be
\label{e:stationakns}
[P_{n+1}(t),\mathfrak{D}(t)] = 0.
\ee
Being commuting ordinary differential operators such a pair implies an algebraic relationship between $P_{n+1}$ and $\mathfrak{D}$ of the type
\be
P_{n+1}^2 = \overset{2n+1}{\underset{j=0}\Pi}(\mathfrak{D}-z_j),
\ee
where $\{z_j\}_{j=0}^{2n+1}\subset\C$ (cf.~\cite[Chapter~3]{gesztesyholden2003}). In particular, they define a hyper-elliptic curve $\mathcal{K}_n$ of (arithmetic) genus $n$, i.e.,
\be
\mathcal{K}_n: \quad w^2 = \overset{2n+1}{\underset{j=0}\Pi}(z-z_j), \quad \{z_j\}_{j=0}^{2n+1}\subset \C,
\ee
and one says $(p,q)$ is an \textit{algebro-geometric} AKNS-solution (cf.~\cite[Chapter~3]{gesztesyholden2003}, and references therein). This establishes an important connection between integrable nonlinear dispersive equations and function theory on Riemann surfaces. 
 
Remarkably, the algebro-geometric AKNS-solutions can be obtained via the IST and expressed in terms of Riemann-theta functions \cite{BBEIM,gesztesyholden2003}. Importantly, Gesztesy and Weikard gave a characterization of all elliptic algebro-geometric AKNS solutions \cite{GW_acta1998, gesztesyweikard_bams1998}. Specifically, elliptic algebro-geometric AKNS solutions correspond to the potentials $(p,q)$ for which the spectral problem $\mathfrak{D}\Psi=z\Psi$ admits a meromorphic fundamental matrix in $x$ for all values of the spectral parameter $z\in\C$.

Additionally, the tools of integrability have recently been applied to curved space-time theories and general relativity. In a recent work it was shown that the AKNS-hierarchy admits applications in general relativity. In particular, it was shown that the AKNS-hierarchy can be obtained as the dynamical system of three-dimensional general relativity with a negative cosmological constant (see \cite{CCLP} for details). 

Importantly, since \eqref{e:dtype} is non-self-adjoint, its spectrum is in general quite difficult to study. Indeed much of the research in this area has been devoted to studying special cases such as exactly-solvable models, potentials with purely imaginary eigenvalues, and algebro-geometric potentials (e.g., see \cite{BLOT,BJ,BM,DM,GW_acta1998,K,KS2002,TV}). Another area of research is the study of \eqref{e:dtype} with the NLS reductions $p=\pm\overline{q}$ in certain asymptotic settings such as the semiclassical limit \cite{BiondiniOregero,bronski,pre2017deng,fujiiewittsten2018,FHK,HK,miller2001}. For general studies of \eqref{e:dtype} see \cite{djakovmityagin,kapmity,kappeler,rofebek,tkachenko} and references therein.
Finally, the study of spectral enclosure estimates for non-self-adjoint operators arising in mathematical physics has a rich history \cite{BOT,bronski,Buckingham,Difranco,Shin} with important applications to the theory of integrable systems, and to the numerical analysis of non-self-adjoint operators.  

\begin{remark}
One can easily verify that $\mathfrak{D}$ is unitarily equivalent to
\be
h\begin{pmatrix} 0 & -1 \\ 1 & 0 \end{pmatrix}\p_x + \frac{1}{2}\begin{pmatrix} (p+q) & i(p-q) \\ i(p-q) & -(p+q) \end{pmatrix}, 
\ee
a form frequently used in the literature.
\end{remark}

In this work we study the $L^2(\R)^2$-spectrum of the periodic non-self-adjoint operator \eqref{e:dtype}. The main results are several new bounds which together enclose the spectrum of \eqref{e:dtype}. Importantly, a sufficient condition for semiclassical spectral confinement to the real axis and a bounded subset of the imaginary axis of the spectral plane is obtained. Finally, an example is considered where the sufficient condition is not satisfied and it is shown in this case that the spectrum does not confine to the real and imaginary axes of the spectral plane in the semiclassical limit $h\to 0^+$. Thus, an interesting open question is whether the sufficient condition is also necessary.
	
\section{Floquet theory}
\label{s:floqtheory}

It is a fortunate fact that the spectral theory of periodic differential operators 
follows from the theory of linear homogeneous ordinary differential equations (ODEs) with periodic coefficients (e.g., see \cite{BES2013}, and \cite[p.~249]{rofebek2}). For a functional-analytic approach see also \cite[Chapter~XIII]{RSIV}.

We consider the spectral problem
\be
\label{e:specprob1}
\mathfrak{D}\Psi = z\Psi, \quad z\in\C.
\ee
From \eqref{e:dtype} it follows that \eqref{e:specprob1} can be rewritten as the singularly perturbed $2\times 2$ first-order system of ODEs with $1$-periodic coefficients, namely, 
\be
\label{e:persys}
h \Psi'(x;z,h) = A(x;z)\Psi(x;z,h), \quad x\in\R, 
\ee
where
\be
A(x;z):= \begin{pmatrix} -iz & q(x) \\ p(x) & iz \end{pmatrix},
\ee
$z\in\C$ is the spectral parameter, $``\,'\,"$ denotes differentiation with respect to the independent variable $x$, and $h>0$ is the semiclassical parameter. We say $\Psi(x;z,h)$ is a solution of \eqref{e:persys}, if $\Psi\in AC_{\rm loc}(\R)^2$ and the equality \eqref{e:persys} holds almost everywhere (a.e.) \cite[Chapter~1]{BES2013}.

Basic properties of the spectrum, denoted $\sigma(\mathfrak{D})$, are reviewed next in order to introduce relevant concepts and notations. 

\begin{theorem}
\label{t:floquet}
(\cite[Chapter~1]{BES2013},\cite{Floquet})
Consider the system of linear homogeneous ODEs
\be
\label{e:floquetODE}
Y' = A(x)Y, \quad x\in\R, 
\ee
where $A:\R\to\C^{n\times n}$ is a locally integrable $n\times n$ matrix-valued function with $A(x+\tau)=A(x)$ for any $x\in\R$. Then any fundamental matrix solution $Y(x)$ of \eqref{e:floquetODE} can be written in the Floquet normal form
\be
\label{e:normalform}
Y(x) = X(x)e^{xR},
\ee 
where $X:\R\to\C^{n\times n}$ is $\tau$-periodic and nonsingular, and $R$ is a $n\times n$ constant matrix. 
\end{theorem}
Without loss of generality one can assume $R$ is in the Jordan normal form. Since $X\in AC_{\rm loc}(\R,\C^{n\times n})$ and $X(x+\tau)=X(x)$, the behavior of solutions as $x\to\pm\infty$ is determined by the eigenvalues, called Floquet exponents, of the constant matrix $R$. In particular: (i) If the Floquet exponent has non-zero real part then solutions grow exponentially as $x\to\infty$, or as $x\to-\infty$; (ii) If the Floquet exponent has zero real part, but $R$ has non-trivial Jordan blocks then the solution is algebraically growing; (iii) Otherwise the Floquet exponent is purely imaginary and the solution remains bounded for all $x\in\R$. 

Consider the $2\times 2$ system given by \eqref{e:persys}. Moreover, let $Y:\R\to\C^{2\times 2}$ be a fundamental matrix of \eqref{e:persys}, that is,
\be
\label{e:ivp}
h Y'(x;z,h) = A(x;z) Y(x;z,h), \quad Y(0;z,h)= Y_o,
\ee
where $Y_o$ is constant. According to Theorem~\ref{t:floquet} it follows that
\[
Y(x+1;z,h) = Y(x;z,h)e^{R(z;h)}.
\]
Hence there exists a non-singular matrix $M\in\C^{2\times 2}$, called a monodromy matrix satisfying
\be
\label{e:monodromy}
M(z;h) = e^{R(z;h)}.
\ee
Importantly, all monodromy matrices are similar. 
Moreover, since \eqref{e:persys} is traceless Abel's theorem implies $\det M(z;h)= 1$. Hence, the eigenvalues of $M(z;h)$ are given by
\be
\label{e:evalues}
\varrho_{1,2} = \Delta \pm \sqrt{\Delta^2-1},
\ee 
where 
\be
\Delta(z;h) := {\rm tr}M(z;h)/2 
\ee
is the \textit{Floquet discriminant} (``${\rm tr}$'' denoting matrix trace). Further, from \eqref{e:evalues} we have $\varrho_1\varrho_2=1$.

In order to describe $\sigma(\mathfrak{D})$ we introduce the \textit{conditional stability set} $\mathcal{S}$,
\be
\label{e:stability}
\mathcal{S} := \Big\{z\in\C: \Delta(z;h) \in [-1,1]\Big\},
\ee
which corresponds to the set of $z\in\C$ such that $\mathfrak{D}\Psi(\cdot;z,h)=z\Psi(\cdot;z,h)$ has at least one non-trivial bounded solution $\Psi(\cdot;z,h)$ on $\R$. Moreover, it follows from Theorem~\ref{t:floquet} that non-trivial solutions of the spectral problem \eqref{e:specprob1} cannot be square-integrable\footnote{In fact, such solutions can not have finite norm in $L^p(\R)^2$ for any $1\le p<\infty$.} over $\R$ and, at best, can be bounded functions on $\R$. Further, Theorem~\ref{t:floquet} and \eqref{e:evalues} together imply that $z\in\mathcal{S}$ if and only if there exists $\xi\in(-\pi,\pi]$\footnote{Note that $i\xi$ is a Floquet exponent} such that
\be
\label{e:boundedsoln}
\Psi(x;z,h) = e^{i\xi x}\Phi(x;z,h), \quad \Phi(x+1;z,h)=\Phi(x;z,h),
\ee
where $\Phi\in L^2(\T)^2$, and $\T:=\R/\Z$ is the circle of length one. 

Next, substituting \eqref{e:boundedsoln} into \eqref{e:specprob1}, it follows that $z\in\mathcal{S}$ if and only if there exists a $\Phi\in L^2(\T)^2$ and a $\xi\in(-\pi,\pi]$ such that
\be
\label{e:specprobxi}
\mathfrak{D}_{\xi}\Phi = z\Phi,
\ee
where
\be
\label{e:blochD}
\mathfrak{D}_{\xi}: H^1(\T)^2\subset L^2(\T)^2\to L^2(\T)^2, \qquad \mathfrak{D}_{\xi} := e^{-i\xi x}\mathfrak{D}e^{i\xi x}.
\ee
($H^k$ the Sobolev space with $k$ distributional derivatives in $L^2$.)

The one-parameter family of operators $\mathfrak{D}_{\xi}$ are called \textit{Bloch operators} associated with $\mathfrak{D}$, and $\xi$ is referred to as the \textit{Bloch frequency}. Since the Bloch operators have compactly embedded domains in $L^2(\T)^2$, for each $\xi\in(-\pi,\pi]$ it follows that the $L^2(\T)^2$-spectrum of $\mathfrak{D}_{\xi}$, denoted $\sigma_{L^2(\T)^2}(\mathfrak{D}_{\xi})$, consists entirely of isolated eigenvalues with finite algebraic multiplicities. Thus, $\sigma(\mathfrak{D})$ admits a continuous parameterization by a one-parameter family of $1$-periodic eigenvalue problems for the associated Bloch operators $\{\mathfrak{D}_{\xi}\}_{\xi\in(-\pi,\pi]}$.

We have the following important result regarding differential operators with periodic coefficients.
\begin{theorem}
\label{t:spectrum}
Let $\mathfrak{D}$ defined by \eqref{e:dtype} be a closed linear operator generated in the space $L^2(\R)^2$. Then the spectrum of $\mathfrak{D}$
\begin{itemize}
\item[(i)] is purely continuous, i.e., essential without any eigenvalues and without residual spectrum.
\item[(ii)] coincides with the conditional stability set $\mathcal{S}$.
\item[(iii)] consists of an infinite number (in general) or finite number of analytic arcs.
\end{itemize}
\end{theorem}
Moreover, it follows that
\be
\sigma(\mathfrak{D}) = \bigcup_{\xi\in(-\pi,\pi]} \sigma_{L^2(\T)^2}(\mathfrak{D}_{\xi}) = \mathcal{S}.
\ee
See \cite{johnson,rofebek} for proofs of this result, and \cite[p.~249]{rofebek2} for a detailed discussion. 

\section{Symmetries}
\label{s:symmetries}

In this section we derive various symmetries for solutions of \eqref{e:persys}. In turn, these results imply certain symmetries in the spectrum of \eqref{e:dtype}.

Let $Y(\cdot;z,h)$ be the canonical fundamental matrix solution of \eqref{e:ivp}, i.e., $Y$ is a fundamental matrix with $Y(0;z,h)= I$ ($I$ denotes the unit $2\times 2$ matrix). Fix $h>0$. First, assume that $p$, $q$ are real. Then 
\be
\label{e:real}
Y_{(r)}(x;z,h) \equiv \overline{Y(x;-\overline{z},h)}
\ee
also is a solution of \eqref{e:ivp}. Next, assume $p$, $q$ are even. Then
\be
\label{e:even}  
Y_{(e)}(x;z,h)\equiv \sigma_3Y(-x;-z,h)
\ee
also is a solution of \eqref{e:ivp}. Finally, assume $p$, $q$ are odd. Then
\be
\label{e:odd}
Y_{(o)}(x;z,h)\equiv Y(-x;-z,h)
\ee
also is a solution of \eqref{e:ivp}.

Next, since the monodromy matrix satisfies $Y(x+1;z,h)=Y(x;z,h)M(z;h)$ it follows that if $p$, $q$ are real, then the monodromy matrix admits the symmetry
\be
M(z;h)= \overline{M(-\overline{z};h)}.
\ee
Similarly, if $p$, $q$ are even, or odd, then the monodromy matrix satisfies the symmetries
\bse
\begin{align}
M(z;h)&= \sigma_3M(-z;h)^{-1}\sigma_3,\\
M(z;h)&= M(-z;h)^{-1},
\end{align}
\ese
respectively.

The above symmetries imply the following result.
\begin{theorem}
If $p$ and $q$ are real, then $\sigma(\mathfrak{D})$ admits a $z\mapsto-\overline{z}$ symmetry. If $p$ and $q$ are even, or odd, then $\sigma(\mathfrak{D})$ admits a $z\mapsto-z$ symmetry. Finally, if $p$ and $q$ are either real and even, or real and odd, then $\sigma(\mathfrak{D})$ is symmetric with respect to both the real and imaginary axes of the spectral variable.  
\end{theorem}

\section{Spectral bounds}
\label{s:estimates}

Next, using simple energy techniques together with Theorem~\ref{t:spectrum} we derive bounds on the $L^2(\R)^2$-spectrum of $\mathfrak{D}$. Importantly, we will see that one of these estimates depends explicitly on the semiclassical parameter $h$. For $z\in\C$ we denote $\Re z$, and $\Im z$, the real and imaginary components of a complex number, respectively. Finally, $L^{\infty}(\R)$ denotes the space of essentially bounded Lebesgue measurable functions.

\begin{lemma}
\label{l:lem1}
Fix $h>0$ and assume $p$, $q \in L^{\infty}(\R)$. If $z\in\sigma(\mathfrak{D})$, then
\be
\label{e:bound1}
|\Im z| \le \left(\|p\|_{\infty}\|q\|_{\infty}\right)^{1/2}.
\ee
\end{lemma}

\begin{proof}
We begin by writing the spectral problem \eqref{e:specprob1} in component form as follows:
\bse
\label{e:sys1}
\begin{align} 
\label{e:comp1}
ih \psi_1'(x;z,h) &=  iq(x)\psi_2(x;z,h) + z\psi_1(x;z,h), \\
\label{e:comp2} 
ih \psi_2'(x;z,h) &= ip(x)\psi_1(x;z,h) - z\psi_2(x;z,h).
\end{align} 
\ese 
Since $z\in\sigma(\mathfrak{D})$ it follows from \eqref{e:boundedsoln} and Theorem~\ref{t:spectrum} that there exists a non-constant solution $\Psi(\cdot;z,h)$ of \eqref{e:sys1} such that $\Psi(1;z,h) = e^{i\xi}\Psi(0;z,h)$ for some $\xi\in(-\pi,\pi]$. Multiply \eqref{e:comp1} by $\overline{\psi}_1$, and integrate by parts to get
\be
\label{e:a}
i\int_0^1q\overline{\psi}_1\psi_2\,dx = -ih\int_0^1\psi_1\overline{\psi}_1'\,dx - z\int_0^1|\psi_1|^2\,dx.
\ee 
Similarly, multiply \eqref{e:comp1} by $\overline{\psi}_1$, take the complex conjugate, and integrate to get
\be
\label{e:b}
i\int_0^1\overline{q}\psi_1\overline{\psi}_2\,dx = ih\int_0^1\psi_1\overline{\psi}_1'\,dx + \overline{z}\int_0^1|\psi_1|^2\,dx.
\ee 
Adding \eqref{e:b} to \eqref{e:a} gives the identity
\be
\label{e:key1}
\int_0^1 q\overline{\psi}_1\psi_2+\overline{q}\psi_1\overline{\psi}_2\,dx = -2\Im z \int_0^1|\psi_1|^2\, dx.
\ee
Hence, by the triangle and H\"older inequalities it follows that
\begin{align*}
2|\Im z|\|\psi_1\|^2_{L^2(0,1)} &= \left| \int_0^1 q\overline{\psi}_1\psi_2+\overline{q}\psi_1\overline{\psi}_2\,dx \right| \\
&\le 2\|q\|_{\infty} \|\psi_1\|_{L^2(0,1)}\|\psi_2\|_{L^2(0,1)}.
\end{align*} 
Thus,
\be
\label{e:est1}
|\Im z|\|\psi_1\|_{L^2(0,1)} \le \|q\|_{\infty}\|\psi_2\|_{L^2(0,1)}.
\ee

Next, multiply \eqref{e:comp2} by $\overline{\psi}_2$, and integrate by parts to get
\be
\label{e:c}
i\int_0^1p\psi_1\overline{\psi}_2\,dx = -ih\int_0^1\psi_2\overline{\psi}_2'\,dx - z\int_0^1|\psi_2|^2\,dx.
\ee 
Similarly, multiply \eqref{e:comp2} by $\overline{\psi}_2$, take the complex conjugate, and integrate to get
\be
\label{e:d} 
i\int_0^1\overline{p}\overline{\psi}_1\psi_2\,dx = ih\int_0^1\psi_2\overline{\psi}_2'\,dx + \overline{z}\int_0^1|\psi_2|^2\,dx.
\ee 
Adding \eqref{e:d} to \eqref{e:c} gives the identity
\be
\label{e:key2} 
\int_0^1p\psi_1\overline{\psi}_2+\overline{p}\overline{\psi}_1\psi_2\,dx = 2\Im z\|\psi_2\|_{L^2(0,1)}^2. 
\ee
Hence, by the triangle and H\"older inequalities it follows that
\begin{align*}
2|\Im z|\|\psi_2\|_{L^2(0,1)}^2 &= \left| \int_0^1 p\psi_1\overline{\psi}_2+\overline{p}\overline{\psi}_1\psi_2\,dx \right| \\
&\le 2\|p\|_{\infty} \|\psi_1\|_{L^2(0,1)}\|\psi_2\|_{L^2(0,1)}. 
\end{align*} 
Thus,
\be
\label{e:est2}
|\Im z|\|\psi_2\|_{L^2(0,1)} \le \|p\|_{\infty}\|\psi_1\|_{L^2(0,1)}.
\ee
Hence, multiplying \eqref{e:est1} and \eqref{e:est2} gives the estimate
\be
|\Im z| \le \left(\|p\|_{\infty}\|q\|_{\infty}\right)^{1/2}.
\ee
\end{proof}

\begin{remark} 
Note that \eqref{e:key1} and \eqref{e:key2} give identities for the imaginary component of the spectral parameter, namely, 
\be
\Im z = -\frac{\Re \langle \psi_1,q\psi_2\rangle}{\langle \psi_1,\psi_1\rangle} = \frac{\Re \langle p\psi_1,\psi_2\rangle}{\langle \psi_2,\psi_2\rangle},
\ee 
respectively. Here $\langle \cdot,\cdot\rangle := \langle \cdot,\cdot\rangle_{L^2(0,1)}$ is the standard inner product associated with the Hilbert space $L^2(0,1)$.
\end{remark} 

\begin{remark}
Note that by applying a similar argument as above to the system \eqref{e:specprob1} one can easily obtain the estimate $|\Im z|\le \max(\|p\|_{\infty},\|q\|_{\infty})$. On the other hand, by considering the component form of \eqref{e:specprob1} we are able to obtain the sharper bound \eqref{e:bound1}.   
\end{remark}

Next, we obtain a second bound on the spectrum which depends explicitly on the semiclassical parameter $h$.

\begin{lemma}
\label{l:lem2}
Fix $h>0$ and assume $p$, $q\in AC_{\rm loc}(\R)$ with $p'$, $q'\in L^{\infty}(\R)$. If $z\in\sigma(\mathfrak{D})$, then
\be
|\Re z||\Im z| \le \frac{h}{2}\left(\|p'\|_{\infty}+\|q'\|_{\infty}\right) + \frac{1}{4}\|\overline{p}\overline{q}-pq\|_{\infty}.
\ee 
\end{lemma}

\begin{proof}
Again, we begin by writing the spectral problem \eqref{e:specprob1} in component form as follows:
\bse
\label{e:eqs}
\begin{align} 
\label{e:eq1}
ih \psi_1'(x;z,h) &=  iq(x)\psi_2(x;z,h) + z\psi_1(x;z,h), \\
\label{e:eq2} 
ih \psi_2'(x;z,h) &= ip(x)\psi_1(x;z,h) - z\psi_2(x;z,h).
\end{align} 
\ese 
Since $z\in\sigma(\mathfrak{D})$ it follows from \eqref{e:boundedsoln} and Theorem \ref{t:spectrum} that there exists a non-constant solution $\Psi(\cdot;z,h)$ of \eqref{e:sys1} such that $\Psi(1;z,h) = e^{i\xi}\Psi(0;z,h)$ for some $\xi\in(-\pi,\pi]$. First, multiply \eqref{e:eq1} by $\overline{\psi}_1'$, and \eqref{e:eq2} by $\overline{\psi}_2'$, and integrate by parts to get
\bse
\begin{align}
\label{e:aa}
ih\int_0^1|\psi_1'|^2\,dx &= -i\int_0^1 (q'\psi_2+q\psi_2')\overline{\psi}_1\,dx + z\int_0^1\psi_1\overline{\psi}_1'\,dx, \\
\label{e:bb}
ih\int_0^1|\psi_2'|^2\,dx &= -i\int_0^1 (p'\psi_1+p\psi_1')\overline{\psi}_2\,dx- z\int_0^1\psi_2\overline{\psi}_2'\,dx.
\end{align}
\ese
Next, multiply \eqref{e:eq1} by $\overline{\psi}_1'$, and \eqref{e:eq2} by $\overline{\psi}_2'$, take the complex conjugate, and integrate by parts to get
\bse
\begin{align}
\label{e:cc}
-ih\int_0^1|\psi_1'|^2\,dx &= i\int_0^1 (\overline{q}'\overline{\psi}_2 + \overline{q}\overline{\psi}_2')\psi_1\,dx - \overline{z}\int_0^1\overline{\psi}_1'\psi_1\,dx, \\
\label{e:dd} 
-ih\int_0^1|\psi_2'|^2\,dx &= i\int_0^1 (\overline{p}'\overline{\psi}_1 + \overline{p}\overline{\psi}_1')\psi_2\,dx + \overline{z}\int_0^1\overline{\psi}_2'\psi_2\,dx.
\end{align}
\ese 
Add \eqref{e:cc} to \eqref{e:aa}, add \eqref{e:dd} to \eqref{e:bb}, and multiply by $h$ to obtain
\bse
\begin{align} 
\label{e:master1}
h\int_0^1 (q'\psi_2\overline{\psi}_1-\overline{q}'\overline{\psi}_2\psi_1)\,dx &= h\int_0^1 (\overline{q}\overline{\psi}_2'\psi_1-q\psi_2'\overline{\psi}_1)\,dx + 2h\Im z\int_0^1\psi_1\overline{\psi}_1'\,dx, \\
\label{e:master2}
h\int_0^1 (p'\psi_1\overline{\psi}_2-\overline{p}'\overline{\psi}_1\psi_2)\,dx &= h\int_0^1 (\overline{p}\overline{\psi}_1'\psi_2 - p\psi_1'\overline{\psi}_2)\,dx -2h\Im z\int_0^1\psi_2\overline{\psi}_2'\,dx.
\end{align}
\ese 
Finally, substituting \eqref{e:eqs} into \eqref{e:master1} gives
\begin{align*}
&h\int_0^1 \overline{q}\overline{\psi}_2'\psi_1-q\psi_2'\overline{\psi}_1\,dx \\
&=\int_0^1\overline{q}\left\{\overline{p}\overline{\psi}_1-i\overline{z}\overline{\psi}_2\right\}\psi_1 - q\left\{p\psi_1+iz\psi_2\right\}\overline{\psi}_1\,dx \\
&= \int_0^1 (\overline{p}\overline{q}-pq)|\psi_1|^2\,dx -i\Re z\int_0^1q\overline{\psi}_1\psi_2 + \overline{q}\psi_1\overline{\psi}_2\,dx + \Im z\int_0^1q\overline{\psi}_1\psi_2-\overline{q}\psi_1\overline{\psi}_2\,dx,
\end{align*}
and
\begin{align*}
2h\Im z\int_0^1\psi_1\overline{\psi}_1'\,dx &= 2\Im z\int_0^1\psi_1\left\{\overline{q}\overline{\psi}_2+i\overline{z}\overline{\psi}_1\right\}\,dx \\
&= 2\Im z\int_0^1\overline{q}\psi_1\overline{\psi}_2\,dx + 2i(\Im z)\overline{z}\|\psi_1\|_{L^2(0,1)}^2.
\end{align*}
Using \eqref{e:key1} and simplifying then gives the identity
\be
\label{e:esteq1}
h\int_0^1 (q'\overline{\psi}_1\psi_2-\overline{q}'\psi_1\overline{\psi}_2)\,dx = 4i\Re z\Im z\|\psi_1\|_{L^2(0,1)}^2 + \int_0^1(\overline{p}\overline{q}-pq)|\psi_1|^2\,dx.
\ee 
Hence, applying the triangle and H\"older inequalities gives
\begin{align*}
4|\Re z||\Im z|\|\phi_1\|_{L^2(0,1)}^2 &= \left|h\int_0^1 q'\overline{\psi}_1\psi_2-\overline{q}'\psi_1\overline{\psi}_2\,dx - \int_0^1(\overline{p}\overline{q}-pq)|\psi_1|^2\,dx\right| \\
&\le
2h\|q'\|_{\infty}\int_0^1|\overline{\psi}_1\psi_2|\,dx + \|\overline{pq}-pq\|_{\infty}\|\psi_1\|_{L^2(0,1)}^2 \\
&\le 2h\|q'\|_{\infty}\|\psi_1\|_{L^2(0,1)}\|\psi_2\|_{L^2(0,1)} + \|\overline{p}\overline{q}-pq\|_{\infty}\|\psi_1\|_{L^2(0,1)}^2.
\end{align*}
Thus,
\be
\label{e:estimate1}
4|\Re z||\Im z|\|\phi_1\|_{L^2(0,1)}^2 \le 2h\|q'\|_{\infty}\|\psi_1\|_{L^2(0,1)}\|\psi_2\|_{L^2(0,1)} + \|\overline{p}\overline{q}-pq\|_{\infty}\|\psi_1\|_{L^2(0,1)}^2.
\ee

Similar to above we now substitute \eqref{e:eqs} into \eqref{e:master2} to obtain 
\begin{align*}
&h\int_0^1 \overline{p}\overline{\psi}_1'\psi_2 - p\psi_1'\overline{\psi}_2\,dx \\
&=\int_0^1\overline{p}\left\{\overline{q}\overline{\psi}_2+i\overline{z}\overline{\psi}_1\right\}\psi_2 - p\left\{q\psi_2-iz\psi_1\right\}\overline{\psi}_2\,dx \\
&= \int_0^1(\overline{p}\overline{q}-pq)|\psi_2|^2\,dx + i\Re z\int_0^1\overline{p}\overline{\psi}_1\psi_2+p\psi_1\overline{\psi}_2 + \Im z\int_0^1\overline{p}\overline{\psi}_1\psi_2-p\psi_1\overline{\psi}_2\,dx,
\end{align*}
and
\begin{align*}
2h\Im z\int_0^1\psi_2\overline{\psi}_2'\,dx &= 2\Im z\int_0^1\psi_2\{\overline{p}\overline{\psi}_1-i\overline{z}\overline{\psi}_2\}\,dx \\
&=2\Im z\int_0^1\overline{p}\overline{\psi}_1\psi_2 - 2i(\Im z)\overline{z}\|\psi_2\|_{L^2(0,1)}^2.
\end{align*}
Using \eqref{e:key2} and simplifying then gives the identity
\be
\label{e:esteq2} 
h\int_0^1 p'\psi_1\overline{\psi}_2-\overline{p}'\overline{\psi}_1\psi_2\,dx
= 4i\Re z\Im z\|\psi_2\|_{L^2(0,1)}^2 + \int_0^1(\overline{p}\overline{q}-pq)|\psi_2|^2\,dx.
\ee 
Hence, applying the triangle and H\"older inequalities gives
\begin{align*}
4|\Re z||\Im z|\|\psi_2\|_{L^2(0,1)}^2 &= \left|h\int_0^1 p'\psi_1\overline{\psi}_2-\overline{p}'\overline{\psi}_1\psi_2\,dx - \int_0^1(\overline{p}\overline{q}-pq)|\psi_2|^2\,dx\right| \\
&\le
2h\|p'\|_{\infty}\int_0^1|\psi_1\overline{\psi}_2|\,dx + \|\overline{pq}-pq\|_{\infty}\|\psi_2\|_{L^2(0,1)}^2 \\
&\le 
2h\|p'\|_{\infty}\|\psi_1\|_{L^2(0,1)}\|\psi_2\|_{L^2(0,1)} + \|\overline{p}\overline{q}-pq\|_{\infty}\|\psi_2\|_{L^2(0,1)}^2.
\end{align*}
Thus,
\be
\label{e:estimate2}
4|\Re z||\Im z|\|\psi_2\|_{L^2(0,1)}^2 \le 2h\|p'\|_{\infty}\|\psi_1\|_{L^2(0,1)}\|\psi_2\|_{L^2(0,1)} + \|\overline{p}\overline{q}-pq\|_{\infty}\|\psi_2\|_{L^2(0,1)}^2.
\ee

Finally, adding \eqref{e:estimate2} to \eqref{e:estimate1}, and applying Young's inequality gives
\begin{align*}
4|\Re z||\Im z|\|\Psi\|_{L^2(0,1)^2}^2 &\le 2h\left(\|p'\|_{\infty}+\|q'\|_{\infty}\right)\|\psi_1\|_{L^2(0,1)}\|\psi_2\|_{L^2(0,1)} + \|\overline{p}\overline{q}-pq\|_{\infty}\|\Psi\|_{L^2(0,1)^2}^2 \\
&\le 2h\left(\|p'\|_{\infty}+\|q'\|_{\infty}\right)\|\Psi\|_{L^2(0,1)^2}^2 + \|\overline{p}\overline{q}-pq\|_{\infty}\|\Psi\|_{L^2(0,1)^2}^2.
\end{align*}
Hence, we arrive at the estimate
\be
|\Re z||\Im z| \le \frac{h}{2}\left(\|p'\|_{\infty}+\|q'\|_{\infty}\right) + \frac{1}{4}\|\overline{p}\overline{q}-pq\|_{\infty}.
\ee
\end{proof}

Importantly, if the potential satisfies the symmetry $\overline{\det Q}=\det Q$, or equivalently, if $\overline{pq}=pq$, then we can obtain a sharper bound. In particular, we have the following estimate.
\begin{lemma}
\label{l:lem3}
Fix $h>0$ and assume $p$, $q\in AC_{\rm loc}(\R)$ with $p'$, $q'\in L^{\infty}(\R)$. Moreover, assume $\overline{p}\overline{q}=pq$. If $z\in\sigma(\mathfrak{D})$, then
\be
\label{e:lem3est}
|\Re z||\Im z| \le \frac{h}{2}\left(\|p'\|_{\infty}\|q'\|_{\infty}\right)^{1/2}.
\ee
\end{lemma}
\begin{proof}
Since $\overline{p}\overline{q}=pq$ it follows that estimate \eqref{e:estimate1} reads
\be
\label{e:newest1}
4|\Re z||\Im z|\|\psi_1\|_{L^2(0,1)}^2 \le 2h\|q'\|_{\infty}\|\psi_1\|_{L^2(0,1)}\|\psi_2\|_{L^2(0,1)},
\ee
and estimate \eqref{e:estimate2} reads
\be
\label{e:newest2}
4|\Re z||\Im z|\|\psi_2\|_{L^2(0,1)}^2 \le 2h\|p'\|_{\infty}\|\psi_1\|_{L^2(0,1)}\|\psi_2\|_{L^2(0,1)}.
\ee
Thus, multiplying \eqref{e:newest1} and \eqref{e:newest2} gives 
\[
16|\Re z|^2|\Im z|^2\|\psi_1\|_{L^2(0,1)}^2\|\psi_2\|_{L^2(0,1)}^2 \le 4h^2\|p'\|_{\infty}\|q'\|_{\infty}\|\psi_1\|_{L^2(0,1)}^2\|\psi_2\|_{L^2(0,1)}^2.
\]
Hence, we have the estimate
\be 
|\Re z||\Im z| \le \frac{h}{2}\left(\|p'\|_{\infty}\|q'\|_{\infty}\right)^{1/2}.
\ee
\end{proof} 

\begin{remark} 
Note that the estimate in Lemma~\ref{l:lem3} is sharper than that of Lemma~\ref{l:lem2} due to the Arithmetic-Geometric Mean inequality, namely,
\be
\label{e:amgm}
(\|p'\|_{\infty}\|q'\|_{\infty})^{1/2} \le \frac{1}{2}\left(\|p'\|_{\infty} + \|q'\|_{\infty}\right).
\ee 
\end{remark} 

The above estimates bound $\sigma(\mathfrak{D})$ to certain closed subsets of the spectral plane. Moreover, when considered together, these estimates provide enclosure estimates on the spectrum as discussed in the next section. 

\section{Enclosure estimates and semiclassical confinement}
\label{s:main}

In this section we present several spectral enclosure estimates for the $L^2(\R)^2$-spectrum of the non-self-adjoint Dirac operator \eqref{e:dtype}. These estimates are relevant to problems in mathematical physics, nonlinear waves, and the numerical study of non-self-adjoint Dirac operators \cite{AKNS,BBEIM,CCLP,gesztesyholden2003}.

\begin{theorem}
\label{t:confine1}
Fix $h>0$ and assume $p$, $q\in AC_{\rm loc}(\R)$ with $p'$, $q'\in L^{\infty}(\R)$. Then
\[
\sigma(\mathfrak{D}) \subset \Lambda^{h}(p,q),
\] 
where
\be
\label{e:lamh}
\Lambda^{h}(p,q) := \Big\{z: |\Im z|\le \left(\|p\|_{\infty}\|q\|_{\infty}\right)^{1/2}\Big\}\, \cap\, \Big\{z: |\Im z||\Re z|\le C(h)\Big\},
\ee
and
\be
\label{e:C1}
C(h) := \frac{h}{2}\left(\|p'\|_{\infty} + \|q'\|_{\infty}\right) + \frac{1}{4}\|\overline{pq}-pq\|_{\infty}.
\ee 
\end{theorem} 

\begin{proof}
If $z\in\sigma(\mathfrak{D})$, then by Lemma~\ref{l:lem1} it follows that $|\Im z|\le \left(\|p\|_{\infty}\|q\|_{\infty}\right)^{1/2}$. Moreover, by Lemma~\ref{l:lem2} it follows that $|\Im z||\Re z|\le C(h)$. Hence, $z\in\Lambda^h(p,q)$. This completes the proof.
\end{proof}

Moreover, if the product of the potentials is real we are able to obtain a sharper enclosure estimate. 

\begin{theorem}
\label{t:confine2} 
Fix $h>0$ and assume $p$, $q\in AC_{\rm loc}(\R)$ with $p'$, $q'\in L^{\infty}(\R)$. Moreover, assume $\overline{p}\overline{q}=pq$. Then
\[
\sigma(\mathfrak{D}) \subset \widetilde{\Lambda}^{h}(p,q),
\]
where 
\be
\label{e:confineset2}
\widetilde{\Lambda}^{h}(p,q):= \Big\{z: |\Im z|\le \left(\|p\|_{\infty}\|q\|_{\infty}\right)^{1/2}\Big\}\, \cap\, \Big\{z:|\Re z||\Im z| \le c(h)\Big\},
\ee
and
\be
\label{e:C2} 
c(h):= \frac{h}{2}\left(\|p'\|_{\infty}\|q'\|_{\infty}\right)^{1/2}.
\ee 
\end{theorem}

\begin{proof}
If $z\in\sigma(\mathfrak{D})$, then by Lemma~\ref{l:lem1} it follows that $|\Im z|\le \left(\|p\|_{\infty}\|q\|_{\infty}\right)^{1/2}$. Moreover, by Lemma~\ref{l:lem3} it follows that $|\Re z||\Im z|\le c(h)$. Hence, $z\in\widetilde{\Lambda}^h(p,q)$. This completes the proof.
\end{proof}

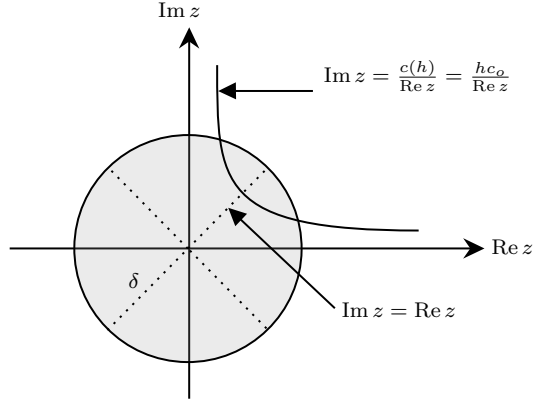
\begin{figure}[t!]
\begin{center}
\tikzset{every picture/.style={line width=0.75pt}}     
\begin{tikzpicture}[x=0.75pt,y=0.75pt,yscale=-1,xscale=1]
\draw    (191,180.59) -- (426,180.59);
\draw [shift={(429,180.6)}, rotate = 180.14] [fill={rgb, 255:red, 0; green, 0; blue, 0 }  ][line width=0.08]  [draw opacity=0] (10.72,-5.15) -- (0,0) -- (10.72,5.15) -- (7.12,0) -- cycle;
\draw    (280.98,256) -- (280.98,72.6);
\draw [shift={(281,69.6)}, rotate = 90.31] [fill={rgb, 255:red, 0; green, 0; blue, 0 }  ][line width=0.08]  [draw opacity=0] (10.72,-5.15) -- (0,0) -- (10.72,5.15) -- (7.12,0) -- cycle;
\draw  [fill={rgb, 255:red, 155; green, 155; blue, 155 }  ,fill opacity=0.2 ] (223.4,179.95) .. controls (223.79,148.49) and (249.61,123.3) .. (281.07,123.69) .. controls (312.53,124.08) and (337.72,149.9) .. (337.33,181.36) .. controls (336.94,212.82) and (311.12,238.01) .. (279.66,237.62) .. controls (248.19,237.23) and (223.01,211.41) .. (223.4,179.95) -- cycle;
\draw  [dash pattern={on 0.84pt off 2.51pt}]  (240,221.6) -- (280.36,180.65);
\draw [color={rgb, 255:red, 0; green, 0; blue, 0 }  ,draw opacity=1 ][line width=0.75]    (295,88.6) .. controls (295,148.6) and (295,171.6) .. (396,171.6);
\draw  [dash pattern={on 0.84pt off 2.51pt}]  (280.36,180.65) -- (320.72,139.71);
\draw  [dash pattern={on 0.84pt off 2.51pt}]  (280.36,180.65) -- (241,140.6);
\draw  [dash pattern={on 0.84pt off 2.51pt}]  (319.72,220.71) -- (280.36,180.65);
\draw    (343,101.56) -- (298,101.56);
\draw [shift={(295.5,101.6)}, rotate = 359.3] [fill={rgb, 255:red, 0; green, 0; blue, 0 }  ][line width=0.08]  [draw opacity=0] (8.93,-4.29) -- (0,0) -- (8.93,4.29) -- cycle;
\draw    (354,210.6) -- (302.73,162.24) ;
\draw [shift={(300.54,160.18)}, rotate = 43.32] [fill={rgb, 255:red, 0; green, 0; blue, 0 }  ][line width=0.08]  [draw opacity=0] (8.93,-4.29) -- (0,0) -- (8.93,4.29) -- cycle;
\draw (249,191.4) node [anchor=north west][inner sep=0.75pt]  [font=\footnotesize]  {$\delta$};
\draw (431,174.4) node [anchor=north west][inner sep=0.75pt]  [font=\footnotesize]  {$\text{Re}\,z$};
\draw (264,55.4) node [anchor=north west][inner sep=0.75pt]  [font=\footnotesize]  {$\text{Im}\,z$};
\draw (347,84.4) node [anchor=north west][inner sep=0.75pt]  [font=\footnotesize]  {$\text{Im}\,z=\frac{c(h)}{\text{Re}\,z}=\frac{hc_o}{\text{Re}\,z}$};
\draw (356,206.4) node [anchor=north west][inner sep=0.75pt]  [font=\footnotesize]  {$\text{Im}\,z =\text{Re}\,z$};	
\end{tikzpicture}
\end{center}
\caption{A disk of radius $\delta$ centered at the origin of the spectral plane together with the curve defined by $G(h):= \{(\Re z, \Im z): \Re z>0,\,\, \Im z = hc_o/\Re z\}$.}
\label{f:cartoon}
\end{figure}

Next, we turn to the enclosure estimate in the semiclassical limit. From \eqref{e:C1} we see that
\[
C(h) = \frac{1}{4}\|\overline{p}\overline{q}-pq\|_{\infty} + O(h) \quad {\rm as} \quad h\to 0^+.
\]
Thus, under the hypotheses of Theorem~\ref{t:confine1} we see that in the semiclassical limit the spectrum is confined to the closed hyperbolic region given by \eqref{e:lamh} with $h=0$. 
	
Importantly, Theorem~\eqref{t:confine2} provides a sufficient condition, namely $\overline{pq}=pq$, in order for the $L^2(\R)^2$-spectrum of \eqref{e:dtype} to confine to the real and imaginary axes in the semiclassical limit $h\to 0^+$ as we show next.

\begin{theorem}
Suppose $p$, $q$ do not depend of $h$.
Assume $p$, $q\in AC_{\rm loc}(\R)$ with $p'$, $q'\in L^{\infty}(\R)$. Moreover, assume $\overline{p}\overline{q}=pq$. Define the set
\be
\label{e:cross}
\Sigma(p,q) := \R \cup i\left[-(\|p\|_{\infty}\|q\|_{\infty})^{1/2},(\|p\|_{\infty}\|q\|_{\infty})^{1/2}\right].
\ee
Moreover, let $N_{\delta}(\Sigma)$ be a $\delta$-neighborhood of $\Sigma(p,q)$. Then for any $\delta>0$, there exists $h_o>0$ such that 
\be
\label{e:inclusion}
\sigma(\mathfrak{D}) \subset N_{\delta}(\Sigma),
\ee
for all $0<h<h_o$.
\end{theorem}

\begin{proof}
The result is a trivial consequence of Theorem~\ref{t:confine2}. A proof is provided for completeness. If $p$, or $q$, are constant then $\sigma(\mathfrak{D})\subset \Sigma(p,q)\subset N_{\delta}(\Sigma)$. 	
Next, fix $\delta>0$.	
Due to the symmetry in \eqref{e:confineset2}, without loss of generality we assume $\Re z>0$, and $\Im z>0$. Define the curve
\be
\label{e:hyperbola}
G(h) := \left\{(\Re z, \Im z): \Re z>0,\,\, \Im z = \frac{hc_o}{\Re z}\right\},
\ee 
where $c_o:= \left(\|p'\|_{\infty}\|q'\|_{\infty}\right)^{1/2}/2>0$. It is easy to show that the distance between the hyperbola $G(h)$ and the boundary of the first quadrant is maximal at the unique point where $G(h)$ intersects with the bisector $\{(\Re z,\Im z): \Re z=\Im z\}$ (see Fig.~\ref{f:cartoon}). Hence, take $h_o= \delta^2/2c_o$. Then if $h<h_o$ at the point of intersection between $G(h)$ and the bisector we have
\[
2(\Im z)^2 = 2hc_o < 2h_oc_o = \delta^2. 
\]
Thus, $\sigma(\mathfrak{D})\subset \widetilde{\Lambda}^{h}(p,q) \subset N_{\delta}(\Sigma)$. This completes the proof.
\end{proof}

\begin{remark}
It is important to point out that the above results also apply to potentials decaying to zero as $x\to\pm\infty$. For simplicity, assume $p$, $q$ are continuously differentiable with $p'$, $q'\in L^{\infty}(\R)$. Moreover, assume $p$, $q \in L^{1,1}(\R)$, where $L^{1,1}(\R) := \{f \in L^1(\R): xf(x) \in L^1(\R)\}$. In this case the above results give confinement estimates for the set of isolated eigenvalues (discrete spectrum). Importantly, in the study of completely integrable nonlinear evolution equations the discrete spectrum parameterizes the so-called ``solitons''--localized traveling wave solutions which interact elastically~\cite{APT2004,NMPZ1984}. Thus it follows that Theorems~\ref{t:confine1}--\ref{t:confine2} provide constraints on the allowable soliton amplitudes and speeds in the solution dynamics. For example, in focusing NLS fast solitons must have sufficiently small amplitudes.  
\end{remark}

Finally, fix $h=1$ and let
\be
\label{e:free}
\mathfrak{D}_o:= i\begin{pmatrix} 1 & 0 \\ 0 & -1 \end{pmatrix}\p_x
\ee
denote the free (massless) Dirac operator. It is easy to see in this case that the $L^2(\R)^2$-spectrum is simply the real $z$-axis. Thus, the spectral enclosure estimates are determined according to the size (measured by the $L^{\infty}(\R)$ norm) and smoothness of perturbations from the free operator \eqref{e:free}. An interesting open question is whether the enclosure estimates can be improved by considering higher-order derivatives of $p$ and $q$.    

\section{Constant potential}
\label{s:example}

In this section, by considering a simple example, we show there exists a potential $Q$ 
of the Dirac operator \eqref{e:dtype} satisfying $\overline{pq}\ne pq$ and such that the spectrum does not confine to the real and imaginary axes of the spectral variable in the semiclassical limit $h\to 0^+$. To this end we consider the case of constant potentials, namely,
\be
\label{e:constpot}
h\Psi'(x;z,h)= \begin{pmatrix} -iz & q_o \\ p_o & iz \end{pmatrix}\Psi(x;z,h), \quad x\in\R,
\ee 
with $p_o\in\C$, and $q_o\in\C$, both constant, $h>0$, and $z\in\C$ the spectral parameter.

Since \eqref{e:constpot} is a $2\times 2$ linear first-order system of ODEs with constant coefficients one can easily compute the solution. In particular, the eigenvalues $\pm i\lambda/h$ corresponding to \eqref{e:constpot} are defined by the complex square root
\be
\lambda(z) = (z^2-\omega^2)^{1/2},
\ee
where $\omega^2 = p_oq_o$. Moreover, the corresponding eigenvector matrix can be written as
\be 
\label{e:eigenvec_mat}
P(z) = \begin{pmatrix} 1 & -\frac{i(z+\lambda)}{p_o} \\ \frac{i(z+\lambda)}{q_o} & 1 \end{pmatrix}.
\ee
Hence, a fundamental matrix solution of \eqref{e:constpot} is given by
\be
\label{e:Fmat}
F(x;z,h) = P(z)e^{i\lambda x\sigma_3/h} = \begin{pmatrix} e^{i\lambda x/h} & \frac{-i(z+\lambda)}{p_o}e^{-i\lambda x/h} \\ \frac{i(z+\lambda)}{q_o}e^{i\lambda x/h} & e^{-i\lambda x/h} \end{pmatrix},
\ee
and the corresponding monodromy matrix is given by
\be
M(z;h) = F^{-1}(0;z,h)F(1;z,h) = \begin{pmatrix} e^{i\lambda/h} & 0 \\ 0 & e^{-i\lambda/h} \end{pmatrix}. 
\ee

Finally, since all monodromy matrices are similar, and the trace of a matrix is invariant under similarity transformations it follows immediately that the Floquet discriminant is given by
\be
\Delta(z) = \cos(\lambda(z)/h),
\ee
where we fixed the period to one.

\begin{figure}[t!]
\centering {\includegraphics[width=6cm]{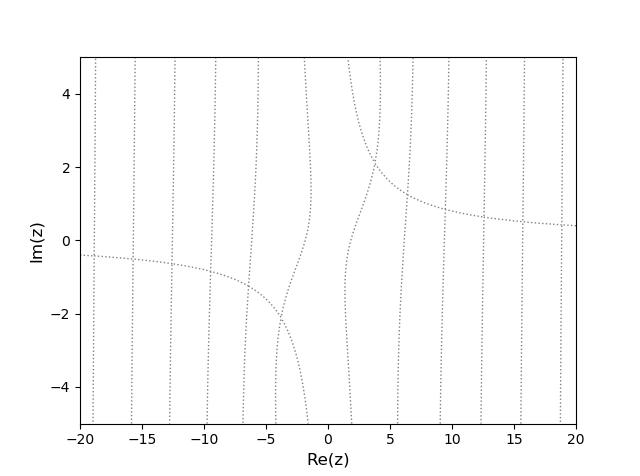}
\includegraphics[width=6cm]{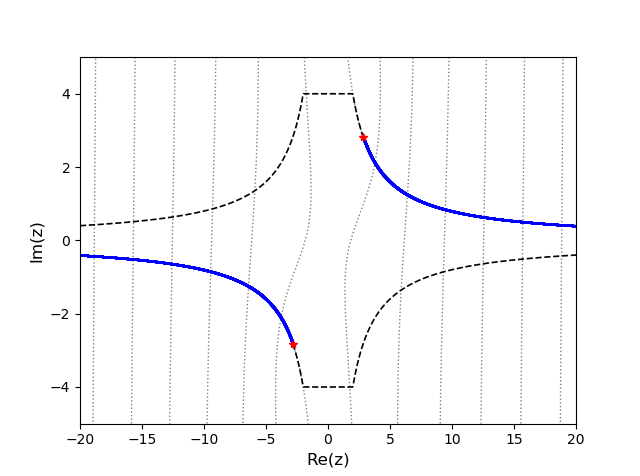}}
\caption{Spectrum of the Dirac operator \eqref{e:dtype} for constant potential with $p_o=1$, $q_o=i16$, and semiclassical parameter $h=1$. Left: Floquet discriminant contours $\Im\Delta=0$. Right: The spectrum, $\sigma(\mathfrak{D})$ (blue), branch points $\pm\omega=2\sqrt{2}(1+i)$ (red star), contours $\Im\Delta=0$ (black dotted), and the boundary of $\Lambda^{h}(p,q)$ (black dashed).}
\label{f:const}
\end{figure}

\begin{figure}[t!]
\centering {\includegraphics[width=6cm]{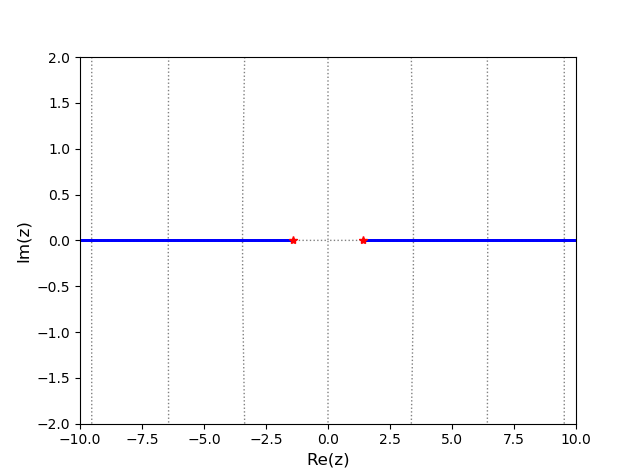}
\includegraphics[width=6cm]{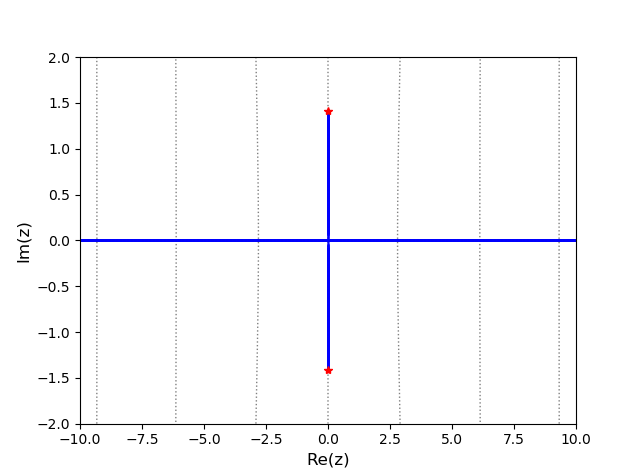}}
\caption{Spectrum of the Dirac operator \eqref{e:dtype} for constant potential and the reductions $p_o=\pm\overline{q_o}$. Left: For $p_o=1+i$, $q_o=1-i$ we plot the spectrum $\sigma(\mathfrak{D})$ (blue), branch points $\pm\sqrt{2}$ (red star), and contours $\Im\Delta=0$ (black dotted). Right: For $p_o=-1-i$, $q_o=1-i$ we plot the spectrum $\sigma(\mathfrak{D})$ (blue), branch points $\pm i\sqrt{2}$ (red star), and contours $\Im\Delta=0$ (black dotted).}
\label{f:constnls}
\end{figure}

Thus, in this case it follows from Theorem~\ref{t:spectrum} that 
\be
\label{e:spectrumconst}
\sigma(\mathfrak{D}) = \Big\{z\in\C: \cos(\lambda(z)/h) \in [-1,1]\Big\}.
\ee 
It follows that the spectrum does not tend to the real and imaginary axes as $h\to 0^+$. In fact for $p_o$, $q_o$ constant $\sigma(\mathfrak{D})$ remains fixed in the complex plane as $h\to 0^+$. For illustrative purposes we compute $\sigma(\mathfrak{D})$ when $p_o=1$, and $q_o=i16$ (see Fig.~\ref{f:const}). 

Finally, the following corollary follows easily. 

\begin{corollary}
\label{c:corone}
Assume $p$, $q\in AC_{\rm loc}(\R)$. If $p$, or $q$, are constant and $\overline{pq}=pq$, then
\be
\sigma(\mathfrak{D}) \subset \Sigma(p,q),
\ee
where $\Sigma(p,q)$ is defined by \eqref{e:cross}.
\end{corollary} 

\begin{proof}
This result follows immediately from Theorem~\ref{t:confine2}.
\end{proof}

As an illustration of Corollary \ref{c:corone} if $p_o=\pm\overline{q_o}$, then 
\bse
\begin{align}
\label{e:dnls}
\sigma(\mathfrak{D})&= \big(-\infty,-|q_o|\big]\cup \big[|q_o|,\infty\big),
\\
\label{e:fnls}
\sigma(\mathfrak{D})&= \R\cup \big[-i|q_o|, i|q_o|\big],
\end{align} 
\ese
respectively (see Fig.~\ref{f:constnls}). Importantly, \eqref{e:dnls} and \eqref{e:fnls} correspond to the continuous spectrum for the defocusing and focusing NLS equations with symmetric non-zero boundary conditions, respectively. Hence, it follows that \eqref{e:spectrumconst} gives the continuous spectrum of the AKNS system with non-zero boundary conditions satisfying $p_{\pm}q_{\pm} = p_oq_o$ as $x\to\pm\infty$, respectively.

\section{Conclusions}

In this work we studied the non-self-adjoint Dirac operator \eqref{e:dtype} with a periodic potential. It was shown that for real and even, or real and odd, potentials the spectrum is symmetric about the real and imaginary axes of the spectral variable. Further, several bounds on the spectrum were obtained for a large class of potentials. The intersection of these bounds provide spectral enclosure estimates. Importantly, it was shown if the determinant of the potential is real, i.e., $\overline{p}\overline{q}=pq$, then no point off the cross \eqref{e:cross} can be in the spectrum if the semiclassical parameter is sufficiently small. Thus, a sufficient condition for semiclassical spectral confinement was obtained.

The results in this work are particularly important for the analysis of the non-self-adjoint Dirac operator \eqref{e:dtype} in the semiclassical limit $h\to 0^+$ using, for example, WKB type methods. The results of this work are also relevant to the study of \eqref{e:dtype} in the case of algebro-geometric potentials, as well as to the numerical computation of the spectrum. An interesting open question is whether the condition $\overline{p}\overline{q}=pq$ is also necessary for semiclassical spectral confinement to the cross \eqref{e:cross}. Finally, another interesting open question is whether one can improve upon the spectral enclosure estimates by restricting to the class of periodic potentials that are $k$-times continuously differentiable with $k\ge 2$. 

\vspace{5mm}

\noindent {\bf Acknowledgments.} The author thanks Gino Biondini, Mathew Johnson, and Dionyssios Mantzavinos for many interesting discussions and several helpful suggestions.

\makeatletter
\makeatother
\def\title#1,{\textit{#1}}
\def\reftitle#1,{``#1''}

\end{document}